\documentclass[11pt]{amsart}
\usepackage{amsfonts}
\usepackage{amsmath}
\usepackage{amsthm}
\usepackage{url}
\usepackage[all]{xy}
\usepackage{latexsym}
\usepackage{amssymb}
\usepackage[cp850]{inputenc}
\usepackage{amsfonts}
\usepackage{graphicx}
\usepackage{dsfont}

\usepackage[usenames]{color}

\newtheorem{sat}{Theorem}[section]		
\newtheorem{lem}[sat]{Lemma}

\newtheorem*{defi*}{Definition}			
\newtheorem*{bei*}{Example}
\newtheorem*{sat*}{Theorem}				
\newtheorem*{kor*}{Corollary}
\newtheorem*{rmk*}{Remark}				
	
\newtheorem*{quest*}{Question}

\let\ssection=\section
\renewcommand{\section}{\setcounter{equation}{0}\ssection}

\newtheorem*{namedtheorem}{\theoremname}
\newcommand{\theoremname}{testing}

\theoremstyle{remark}
\newtheorem*{bem}{Remark}

\newtheorem{claim}{Claim}

\newtheorem*{namedtheoremr}{\theoremnamer}
\newcommand{\theoremnamer}{testing}

\newcommand{\BR}{\mathbb R}			
\newcommand{\BN}{\mathbb N}			
			\newcommand{\BZ}{\mathbb Z}

\newcommand{\CA}{\mathcal A}		
		
		\newcommand{\CF}{\mathcal F}

		\newcommand{\CX}{\mathcal X}
\newcommand{\CY}{\mathcal Y}

\newcommand{\D}{\partial}

\DeclareMathOperator{\SL}{SL}		
\DeclareMathOperator{\GL}{GL}		
\DeclareMathOperator{\Id}{Id}		

\newcommand{\comment}[1]{}

\DeclareMathOperator{\Stab}{Stab}

\DeclareMathOperator{\Fix}{Fix}

\DeclareMathOperator{\Aff}{Aff}

\newcommand{\fsubd}{\mathrel{{\scriptstyle\searrow}\kern-1ex^d\kern0.5ex}}
\newcommand{\bsubd}{\mathrel{{\scriptstyle\swarrow}\kern-1.6ex^d\kern0.8ex}}

\begin{document}

\title[]{Cyclic hyperbolic Veech groups in finite area}
\author{Anna Lenzhen}
\address{IRMAR, Universit\'e de Rennes 1}
\email{anna.lenzhen@univ-rennes1.fr}
\author{Juan Souto}
\address{IRMAR, Universit\'e de Rennes 1}
\email{juan.souto@univ-rennes1.fr}
\thanks{This material is based upon work supported by the National Science Foundation under Grant No. DMS-1440140 while the authors were in residence at the Mathematical Sciences Research Institute in Berkeley, California, during the Fall 2016 semester.}

\begin{abstract}
We prove that there are finite area flat surfaces whose Veech group is an infinite cyclic group consisting of hyperbolic elements.
\end{abstract}
\maketitle

\section{}

A translation surface is an orientable 2-dimensional manifold $X$ with an atlas such that, off a discrete set of points called singularities, transition maps are translations in $\BR^2$. Note that the regular part of such a surface $X$, that is the complement of the set of singularities, is canonically endowed with a flat metric. An affine homeomorphism $f:X\to X$ of a translation surface is an orientation preserving self-homeomorphism which, in charts, is the restriction of an affine self-map of $\BR^2$. Every affine homeomorphism has a well-defined linear part, its differential $Df\in\GL_2\BR$, and if we denote by $\Aff(X)$ the group of affine homeomorphisms we have the homomorphism
$$\Aff(X)\to\GL_2^+\BR,\ \ f\mapsto Df.$$
The image of this homomorphism is the {\em Veech group} of $X$. 

The Veech group of a flat surface is in general trivial but can also be rather large. For instance, so called square tiled surfaces, such as the square torus, have Veech groups commensurable to $\SL_2\BZ$. Veech himself \cite{Veech} constructed other examples of translation surfaces whose Veech group is a lattice in $\SL_2\BR$, and it is altogether an interesting question to figure out what kind of groups can arise as Veech groups of translation surfaces. 

If there are no restrictions on the translation surface $X$, then there is pretty much a complete answer to this problem. For example, it is proved in \cite{Veech infinite} that every countable subgroup of $\GL_2^+\BR$ which does not contain contracting elements is the Veech group of some translation surface. However, once one starts limiting the class of surfaces under consideration new restrictions arise on which subgroups  can be  Veech groups. For example, if the surface $X$ in question has finite area with respect to the canonical flat riemannian metric, then the Veech group is contained in $\SL_2\BR$. If moreover $X$ is not a cylinder and if its regular part contains a periodic geodesic, then the Veech group is in fact a discrete subgroup of $\SL_2\BR$ \cite{Bowman}. In general, the stronger the conditions on the surface are, the more stringent are the restrictions on the possible Veech group, and the less is known about which groups arise. For example, the well-known problem of deciding if there is a compact translation surface whose Veech group consists of a single cyclic group generated by a hyperbolic element is still  open. The goal of this note is to prove that such translation surfaces exist in the class of finite area translation surfaces.

\begin{sat}\label{main}
There are finite area translation surfaces whose Veech group is a cyclic hyperbolic subgroup.
\end{sat}

The surface $X'$ constructed to prove Theorem \ref{main} will arise as a subsurface of the well-known Chamanara surface $X$ \cite{Chamanara}. The later is a non-compact finite area translation surface with a large group of affine transformations. The basic idea of our construction is to choose an element $\phi\in\Aff(X)$ representing a hyperbolic element in the Veech group and find a point $\zeta\in X$ whose stabiliser in $\Aff(X)$ is trivial and whose orbit $\{\phi^n(\zeta)\vert n\in\BZ\}$ under $\phi$ is discrete. Once we have such a point, the surface $X'=X\setminus\{\phi^n(\zeta)\vert n\in\BZ\}$ does the job.

\medskip

We thank Kasra Rafi and Anja Randecker for telling us about the existence of the Chamanara surface, and we thank Reza Chamanara for coming up with it. 

\section{}

\subsection{The Chamanara surface}
In this section we recall the construction of the Chamanara surface $X$ (see \cite{Chamanara} and \cite{Anja} for details).  We start with the closed unit square $[0,1]\times[0,1]$ in $\BR^2$ and consider for $k=0,1,2,\dots$ the intervals
\begin{align*}
I_k^0=&[1-2^{-k},1-2^{-k-1}]\times\{0\}\subset[0,1]\times\{0\},\\
I_k^1=&[2^{-k-1},2^{-k}]\times\{1\}\subset\{1\}\times[0,1],\\
J_k^0=&\{0\}\times[1-2^{-k},1-2^{-k-1}]\subset\{0\}\times[0,1],\text{ and}\\
J_k^1=&\{1\}\times[2^{-k-1},2^{-k}]\subset\{1\}\times[0,1]
\end{align*}
contained in the boundary of the square. Let 
$$Q=[0,1]\times[0,1]\setminus\left(\{(0,1),(1,0)\}\cup\left(\cup_{k=0}^\infty(\D I_k^0\cup\D I_k^1\cup \D J_k^0\cup\D J_k^1\right)\right)$$
be the subset of the plane obtained by removing from the closed unit square the corners and the boundary points of the segments $I_k^0,I_k^1,J_k^0$ and $J_k^1$ for all $k$. The Chamanara surface $X$ is obtained from $Q$ when we identify for all $k$ the intervals $I_k^0$ and $I_k^1$ and the intervals $J_k^0$ and $J_k^1$ via translations. See figure \ref{surface} for a schematic representation of $X$.

\begin{figure}
    \begin{center}
    \centerline{\includegraphics[scale=0.8]{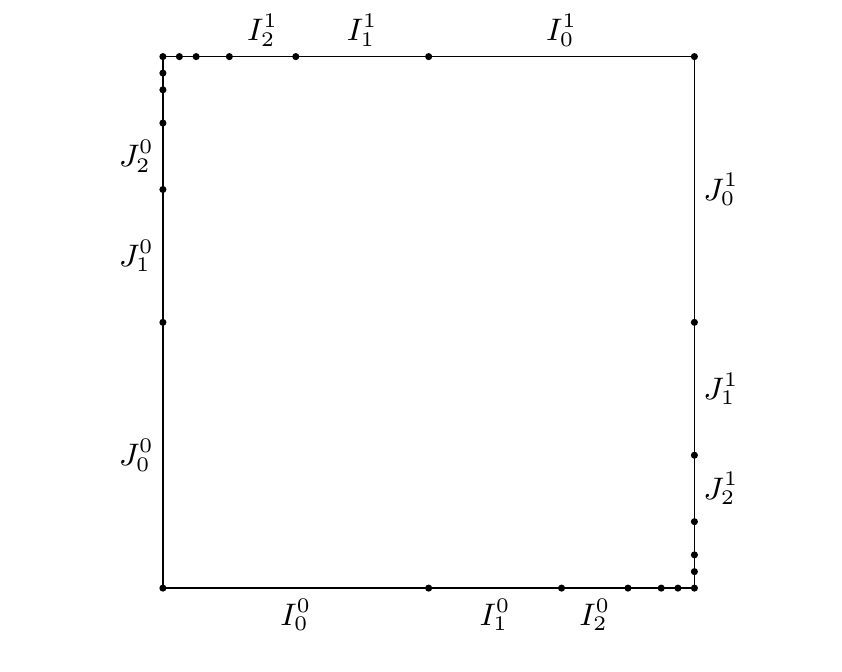}}
    \end{center}
    \caption{Glue by translations parallel segments of the same lengths.}\label{surface}
    \end{figure} 

\subsection{The map $\phi$}
The Chamanara surface $X$ has a non-trivial Veech group. In fact it is a Fuchsian group of second type. However, we will be mostly interested in the individual element
$$\phi:X\to X$$ 
which geometrically can be described as follows. Start with the map $L:\BR^2\to\BR^2$ given by $L(x,y)=\left(\frac 12x+\frac 12,2y\right)$. 
\begin{figure}
    \begin{center}
    \centerline{\includegraphics[scale=0.7]{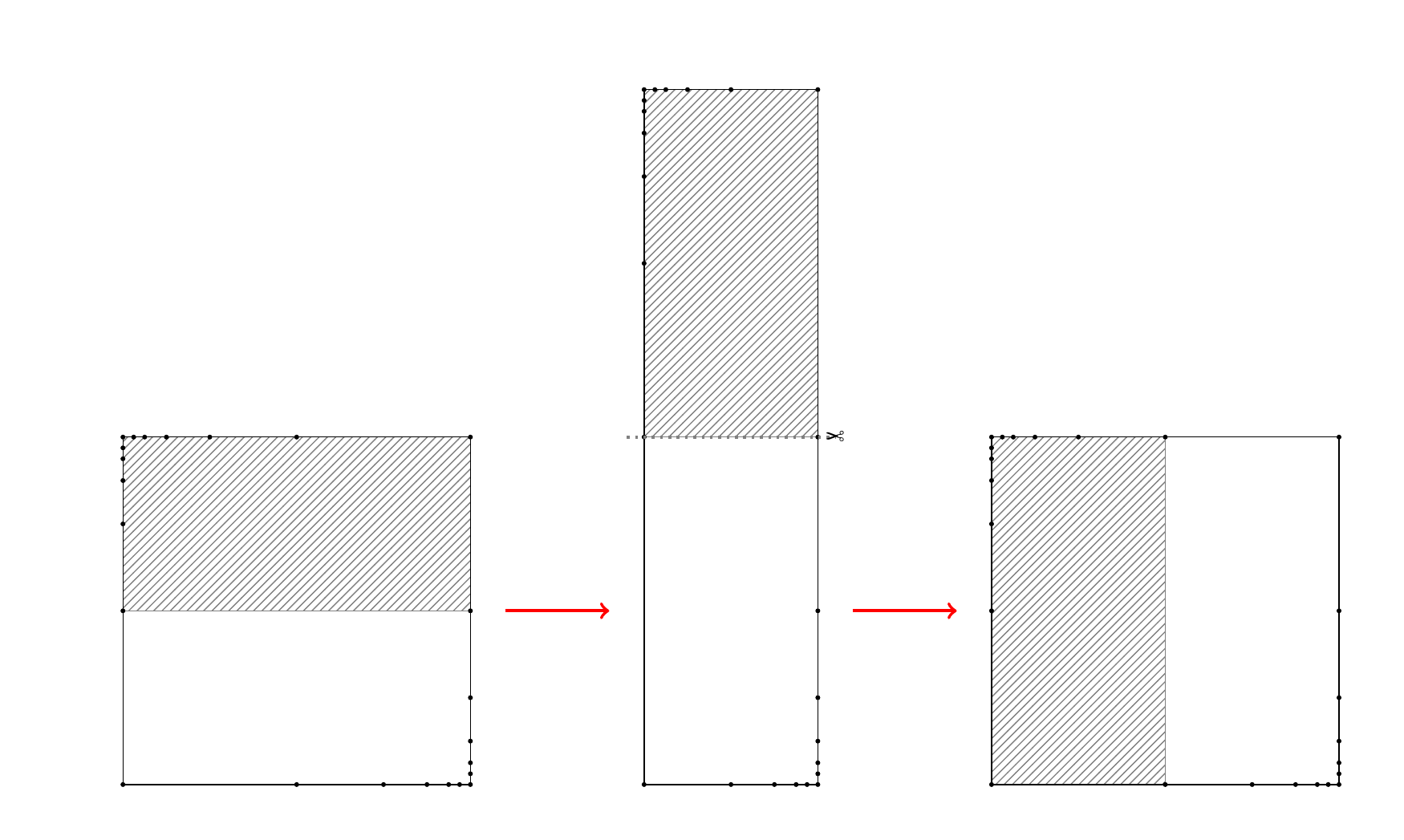}}
    \end{center}
    \caption{Map by $L$, cut along the dotted line, translate and reglue.}\label{map}
    \end{figure} 
Cut the image $L(Q)$ of $Q$ under this map into the two pieces
$$Q_1=L(Q)\cap\left[\frac 12,1\right]\times[0,1]\text{ and }Q_2=L(Q)\cap\left[\frac 12,1\right]\times[1,2]$$
and note that we can reassemble $Q$ from $Q_1$ and $Q_2$ if we identify the latter with $Q\cap\left[0,\frac 12\right]\times[0,1]$ via a translation (see figure \ref{map}). Here is a precise formula for $\phi$. Every $\zeta\in X$ is represented by a unique point $(x,y)\in Q\cap([0,1)\times[0,1))$. The point $\phi(\zeta)$ is just the point in $X$ represented by $(x',y')=\phi(x,y)$, and the later is given by:
\begin{align}\label{imagexy}
&x'=
\begin{cases}
\frac 12x+\frac 12  & \text{if }y\in[0,\frac 12)  \\
\frac 12x  & \text{if }y\in[\frac 12,1).   
\end{cases}
&y'=
\begin{cases}
2y  & \text{if }y\in[0,\frac 12)\\
2y-1  & \text{if }y\in[\frac 12,1)   
\end{cases}
\end{align}
The reader can check that the induced map $\phi:X\to X$ is well-defined and affine. The associated element of the Veech group of $X$ is 
$$d\phi=
\left(
\begin{array}{cc}
\frac 12  & 0  \\
0 & 2  
\end{array}
\right)$$
and thus hyperbolic.

\begin{bem}
Note that the inverse of $\phi$ is given by the same formula once we interchange the roles of $x$ and $y$. To be precise note that the self map of $Q$ given by $(x,y)\mapsto(y,x)$ induces a well defined homeomorphism $\tau:X\to X$ with $\phi^{-1}=\tau\circ\phi\circ\tau$.
\end{bem}

\subsection{Discrete orbits}
We are going to be interested in points $\zeta\in X$ whose orbit $\{\phi^n(\zeta)\vert n\in\BZ\}$ under $\phi$ is discrete. To that end consider the set
$$\CA=\left\{(s_n)_n\in\BN^\BN\vert \lim_n(s_{n+1}-s_n)=\infty\right\}$$
of all sequences $(s_n)$ of natural numbers such that the associated sequence of increments $(s_{n+1}-s_n)$ tends to infinity. Now consider the subsets
$$\CX=\left\{1-\sum_{n\in\BN}2^{-s_n}\middle\vert (s_n)\in\CA\right\}\text{ and }\CY=\left\{\sum_{n\in\BN}2^{-s_n}\middle\vert(s_n)\in\CA\right\}$$
of $(0,1)$. Note that $\CX\times\CY$ is contained in the open square $(0,1)\times(0,1)$ and that $\Phi(\CX\times\CY)=\CX\times\CY$. It follows that we can always use the formulas we gave above for $\Phi$ to compute the orbit of $(x,y)\in\CX\times\CY$.

\begin{lem}\label{lem1}
The orbit $(\phi^n(x,y))_{n\in\BZ}$ of any $(x,y)\in\CX\times\CY$ is discrete in $X$.
\end{lem}
\begin{proof}
Fix $(x,y)\in \CX\times\CY$. We will give a proof of discreteness of the forward orbit. The proof for the backward orbit is almost identical and we leave it to the interested reader. 

We identify points in the orbit of $(x,y)$ with points in $Q$. Suppose then that a subsequence $\phi^{n_k}(x,y)$ of the forward orbit of $(x,y)$ converges in the closed square $[0,1]\times[0,1]$ to some $(\check x,\check y)$. To prove the lemma, it suffices to show that $(\check x, \check y)\notin Q.$ In fact we will prove the following.

\begin{claim}
The point $(\check x, \check y)$ is either of the form $(1-2^{-j},0)$ or of the form $(2^{-j},1)$ for some $j\in \mathbb N\cup\{\infty\}$ where $2^{-\infty}=0$.
\end{claim}

To prove the claim we will use binary expansions of the coordinates. To be precise, we write $z=[z_1,z_2,z_3,\ldots]$ for $z=\sum_{i\geq 1}z_i2^{-i}\in[0,1]$. The coefficients are well defined unless $z$ is a dyadic rational, that is unless it can be written as a finite sum of powers of 2. 

Let 
\begin{align}
&x=[x_1,x_2,x_3,x_4,x_5,\ldots]\\ \nonumber
&y=[y_1,y_2,y_3,y_4,y_5,\ldots]
\end{align}
be the binary expansion of $(x,y)$ and note that we get from \eqref{imagexy} that the expansion of $(x',y')=\phi(x,y)$ is
\begin{align}\label{imagedyadic}
&x'=[1-y_1,x_1,x_2,x_3,x_4,\ldots]\\ \nonumber
&y'=[y_2,y_3,y_4,y_5,y_6,\ldots ]
\end{align}
Equation \eqref{imagedyadic} shows that $(\check x, \check y)$ depends only on $y$ and the sequence $(n_k)$.

Starting with the proof of the claim, observe that passing to a further subsequence we can in fact assume that the binary expansions of the coordinates of $\phi^{n_k}(x,y)$ converge to binary expressions 
 \begin{align}
\check x &=[\check x_1,\check x_2,\check x_3,\check x_4,\check x_5,\ldots]\\ \nonumber
\check y &=[\check y_1,\check y_2,\check y_3,\check y_4,\check y_5,\ldots]
\end{align} 
of the coordinates of $(\check x,\check y)$. Here, convergence of the binary expansions of  coordinates of $\phi^{n_k}(x,y)$ to that of $(\check x,\check y)$ simply means that given any $m\in \mathbb N$, from some $k_0$ on the first $m$ coefficients in the binary expansion
of both coordinates of $\phi^{n_k}(x,y)$  agree with those of $(\check x,\check y)$. 

Let $$K=\text{card}\{i\in \mathbb N \vert \,\check x_i=0\}+\text{card}\{i\in \mathbb N \vert\, \check y_i=1\}.$$ To prove the claim it suffices  to show that $K\leq 1$. In fact, if this were not that case, then there would be some $j$ so that the binary expansion of $y$ had infinitely many segments of the form $1,\underset{j \text{ zeros }}{\underbrace{0,0,\ldots,0}},1$. But this is impossible since by definition of the set $\CY$,  the distance between consecutive 1's in the binary expansion of $y$ goes to infinity. This finishes the proof of the Lemma. 
\end{proof}

\section{Proof of the main theorem}
In this section we prove Theorem \ref{main}. First we will need the following surely well-known lemma:

\begin{lem}\label{lem2}
Let $Y$ be a connected translation surface and $f:Y\to Y$ a non-trivial affine map. The fixed point set $\Fix(f)$ of $f$ consists of a (possibly empty or finite) union of isolated points and of a (again possibly empty or finite) union of isolated arcs.
\end{lem}
\begin{proof}
To begin with recall that the set of singularities of $Y$ is discrete and invariant under any affine transformation. We can thus forget the singularities and assume that every point in $Y$ was regular to begin with. Note also that it suffices to prove that every point $\zeta\in Y$ fixed under $f$ has a small neighborhood such that the set $B\cap\Fix(f)$ is equal to either $\{\zeta\}$ itself or to a straight arc. Continuing with the same notation, suppose that $\epsilon>0$ is small enough so that the metric ball $B(\zeta,\epsilon,Y)$ in $Y$, centered at $\zeta$ and with radius $\epsilon$ is isometric to a round $\epsilon$-ball in $\BR^2$. Let also $\lambda\ge 1$ be a Lipschitz constant for $f$ and suppose that $\zeta'\in B(\zeta,\frac 1{2\lambda}\epsilon,Y)$ is fixed by $f$. Let $[\zeta,\zeta']$ be the straight segment joining $\zeta$ and $\zeta'$ within $B(\zeta,\epsilon,Y)$. Then $f[\zeta,\zeta']$ is another such segment, which means that $f[\zeta,\zeta']=[\zeta,\zeta']$ and thus that $[\zeta,\zeta']$ is fixed pointwise. 

Suppose that $\zeta''\in B(\zeta,\frac 1{2\lambda}\epsilon,Y)$ is another fixed point. The argument we just used, proves that that $[\eta',\eta'']\subset\Fix(f)$ for any two $\eta'\in[\zeta,\zeta']$ and $\eta''\in[\zeta,\zeta'']$ sufficiently close to $\zeta$. Now, either $[\zeta,\zeta']$ and $[\zeta,\zeta'']$ are contained in a segment, or the union of all those segments $[\eta',\eta'']$ has non-empty interior. In the later case we get that the affine map $f$ fixes an open set. By analytic continuation this means that $f=\Id$, which contradicts our assumption.
\end{proof}

We can now prove our main result:

\begin{proof}[Proof of Theorem \ref{main}]
Recall the construction of the Chamanara surface $X$ given before; as above, we will identify points in $Q$ with the corresponding classes in $X$.

It follows for example from \cite{Chamanara,Anja} (or also from \cite{Bowman}) that the group $G$ of affine transformations of the Chamanara surface $X$ is countable. It thus follows from Lemma \ref{lem2} that the set $\CF=\{\zeta\in X\vert\Stab_G(\zeta)\neq\Id\}$ consists of countably many points and countably many straight segments. It thus follows that there are only countably many points $x\in(0,1)$ such that the vertical line $\{x\}\times(0,1)$ contains a non-degenerate interval contained in $\CF$. Since $\CX$ is uncountable, we can thus choose $x\in\CX$ such that $(\{x\}\times(0,1))\cap\CF$ is a countable union of points. Since $\CY$ is also uncountable we get that there is $\zeta=(x,y)\in\CX\times\CY\setminus\CF$.

Consider the complement $X'=X\setminus\{\phi^n(\zeta)\vert n\in\BZ\}$ of the orbit of $\zeta$ under $\phi$ in the Chamanara surface $X$. Since the orbit of $\zeta$ is discrete by Lemma \ref{lem1}, we get that $X'$ is an open subset of a flat surface, and thus a flat surface in its own right. Since the orbit of $\phi$ preserves $X'$, we get that $\phi$ belongs to the Veech group of $X'$. We claim that there is no other element. 

Well, the surface $X'$ has infinitely many ends. One is wild: the end of the Chamanara surface. And the others, the ones produced by removing the orbit, are analytically finite. Any affine map $\psi:X'\to X'$ must then map analytically finite ends to analytically finite ends. Since the analytically finite ends of our surface $X'$ have a neighborhood where the flat structure is isometric to $D\setminus\{0\}$ where $D$ is a small neighborhood of $0$, then we get that $\psi$ extends to a an affine transformation of $X$ which we again denote by $\psi:X\to X$. By construction, and this is is key, $\psi$ maps $\{\phi^n(\zeta)\vert n\in\BZ\}$ into itself. Indeed, suppose  that $\psi(\zeta)=\phi^n(\zeta)$. Then $\phi^{-n}\circ\psi\in\Stab_G(\zeta)$. By the choice of $\zeta$ this means that $\phi^{-n}\circ\psi=\Id$ which means that $\psi=\phi^n$. We have proved that the full group of automorphisms of $X'$ is the cyclic group $\langle\phi\rangle$. We are done.
\end{proof}

\end{document}